\documentclass[11pt,reqno]{amsart}

\usepackage{amssymb,amsmath,amsthm,mathtools}
\usepackage{tikz,tikz-cd}
\usetikzlibrary{intersections}
\usepackage[T1]{fontenc}
\usepackage{palatino}
\usepackage{microtype}
\usepackage[mathscr]{euscript}
\usepackage{stmaryrd}
\usepackage[pagebackref]{hyperref}
\hypersetup{
	colorlinks=true,
	linkcolor=blue,
	urlcolor=blue,
	citecolor=blue,
	linktocpage
}
\usepackage[alphabetic,nobysame]{amsrefs}

\swapnumbers
\newtheorem{thm}{Theorem}[section]
\newtheorem{lem}[thm]{Lemma}
\newtheorem{prop}[thm]{Proposition}

\theoremstyle{definition}
\newtheorem{defn}[thm]{Definition}
\newtheorem{rmk}[thm]{Remark}
\newtheorem{ex}[thm]{Example}

\newcommand{\HH}{H_{n,p}}
\newcommand{\VV}{V_{n,p}}

\newcommand{\C}{\mathbb{C}}
\newcommand{\Q}{\mathbb{Q}}
\newcommand{\TL}{\operatorname{TL}}
\newcommand{\T}{\operatorname{T}}
\newcommand{\ov}{\overline}

\newcommand{\epic}{
	\begin{tikzpicture}[scale = 0.35,thick, baseline={(0,-1ex/2)}]
		\tikzstyle{vertex} = [shape = circle, minimum size = 4pt, inner sep = 1pt]
		\node[vertex] (G--2) at (1.5, -1) [shape = circle, draw,fill=black] {};
		\node[vertex] (G--1) at (0.0, -1) [shape = circle, draw,fill=black] {};
		\node[vertex] (G-1) at (0.0, 1) [shape = circle, draw,fill=black] {};
		\node[vertex] (G-2) at (1.5, 1) [shape = circle, draw,fill=black] {};
		\draw (G--2) .. controls +(-0.5, 0.5) and +(0.5, 0.5) .. (G--1);
		\draw (G-1) .. controls +(0.5, -0.5) and +(-0.5, -0.5) .. (G-2);
	\end{tikzpicture}
}

\newcommand{\onepic}{
	\begin{tikzpicture}[scale = 0.35,thick, baseline={(0,-1ex/2)}]
		\tikzstyle{vertex} = [shape = circle, minimum size = 4pt, inner sep = 1pt]
		\node[vertex] (G--1) at (0.0, -1) [shape = circle, draw,fill=black] {};
		\node[vertex] (G-1) at (0.0, 1) [shape = circle, draw,fill=black] {};
		\draw (G-1) .. controls +(0, -1) and +(0, 1) .. (G--1);
	\end{tikzpicture}
}

\allowdisplaybreaks
\title[The Center of the Temperley--Lieb Algebra]{The Center of the Temperley--Lieb Algebra}
\author{Anthony Giaquinto}
\address{\parbox{\linewidth}{Department of Mathematics and Statistics,
		Loyola University Chicago,\\ Chicago, IL 60660 USA}}
		\email{agiaqui@luc.edu}
\author{Mitja Mastnak}
\address{\parbox{\linewidth}{Department of Mathematics and Computing Science,
		Saint Mary's University,\\ Halifax, NS, Canada B3H C3C}}
	\email{mitja.mastnak@smu.ca}
	\thanks{The second author was supported, in part, by NSERC (Canada).}

\begin{document}

\begin{abstract}
We compute the dimension of the center of the Temperley--Lieb algebra
\(\TL_n(\delta)\) over a field of characteristic zero for every nonzero value
of the parameter \(\delta\). The proof uses the cellular filtration by cup number, together with known facts about the representation theory of the Temperley--Lieb algebra, especially
the structure of its standard modules and their radicals. Dilation and compression
maps compare the induced graded pieces of the center at levels \(n\) and
\(n-2\), giving an upper bound of one for each such piece. A deformation argument gives the matching lower bound, and hence
\[
\dim Z(\TL_n(\delta))=1+\Bigl\lfloor \frac{n}{2}\Bigr\rfloor
\qquad (\operatorname{char}\Bbbk=0,\; \delta\neq 0).
\]
We also prove that every central element is fixed by the canonical
anti-automorphism and by the natural diagram-reflection automorphism. Finally,
we give a congruence criterion for the trivial-radical case and record a
Gram-matrix computation for leading terms.
\end{abstract}
\maketitle
\tableofcontents
\section{Introduction}

The Temperley--Lieb algebra was introduced by Temperley and Lieb in their
study of lattice models in statistical mechanics \cite{TemperleyLieb}. It
later became a fundamental object in subfactor theory, knot theory, and
diagrammatic representation theory. Jones's seminal work on subfactors showed
that the Temperley--Lieb relations are deeply connected with operator
algebras, braid group representations, and low-dimensional topology, leading
to the Jones polynomial \cite{JonesSubfactors}. Kauffman's bracket model
supplied one of the standard planar diagrammatic realizations of these
algebras \cite{KauffmanState}. The Temperley--Lieb algebras are also
prototypical examples of cellular algebras in the sense of Graham and Lehrer
\cite{GrahamLehrer}; the cellular filtration and the associated standard
modules play a central role in what follows. For detailed accounts of the
algebraic, diagrammatic, and historical aspects of Temperley--Lieb algebras,
see Ridout--Saint-Aubin \cite{RidoutSaintAubin} and Doty--Giaquinto
\cite{DotyGiaquinto}.

In characteristic zero, much of the structure of the Temperley--Lieb algebras
is known. At generic parameters \(\TL_n(\delta)\) is split semisimple, while
at roots of unity its representation theory is governed by the standard
modules and their radicals. After writing \(\delta=q+q^{-1}\), the usual generic
case corresponds to \(q\) not being a root of
unity. More precisely, for \(\TL_n(\delta)\), non-semisimplicity can occur only
when \(\ell\le n\), where \(\ell\) is the smallest positive integer such that
\(q^{2\ell}=1\). Thus roots of unity with \(\ell>n\) behave generically for
\(\TL_n(\delta)\). The root-of-unity case, including the non-semisimple
structure, was carefully analyzed by Goodman and Wenzl
\cite{GoodmanWenzl}. The present paper uses most directly the diagrammatic
formulation of Ridout and Saint-Aubin \cite{RidoutSaintAubin}, who give a
detailed treatment of the standard modules and their radicals in the
non-generic case.

The algebra \(\TL_n(\delta)\) also has a representation-theoretic realization
as a centralizer algebra for the action of the quantum group
\(U_q(\mathfrak{sl}_2)\) on tensor powers of its natural two-dimensional
module. This gives another approach to the structure of the Temperley--Lieb algebra through Schur--Weyl duality and tilting modules; see,
for example, Andersen \cite{AndersenTL}. For related categorical treatments at
roots of unity involving fusion categories, see Iohara--Lehrer--Zhang
\cite{ILZJonesQuotient}.

In this paper, we consider one aspect of the structure of the
Temperley--Lieb algebra, namely its center. In the split semisimple case, the
center has one primitive central idempotent for each simple summand, so its
dimension is equal to the number of simple modules. In the non-generic case,
this description is no longer available, and other means are needed to study
the center. In this setting, the center is closely related to the block
structure of the algebra and to the way the radical enters its representation
theory. 

A number of important central elements in Temperley--Lieb algebras are known.
These include full twists descending from braid groups, Jucys--Murphy-type
elements descending from Hecke algebras, central elements such as the element
\(F_n\) used in \cite{RidoutSaintAubin}, and constructions involving
Jones--Wenzl idempotents. These constructions produce important central elements and related families. However, they do not by themselves give a uniform computation of the dimension
of the center. Related questions about centers have also been studied for affine and affine
nil Temperley--Lieb algebras; see \cite{VlasenkoCenterQuotients} and
\cite{BenkartMeinelAffineNilCenter}.

The main result of the present paper is a direct
computation of the dimension of the center \(Z(\TL_n(\delta))\) for all
\(n\) and all \(\delta\in \Bbbk^{\times}\) when the ground field \(\Bbbk\) has characteristic zero. Although \(\TL_n(\delta)\) can become nonsemisimple at roots of unity, the
result shows that, for fixed \(n\), the dimension of the center is constant as
\(\delta\) ranges over all nonzero values. Specifically, we prove that
\[
\dim_{\Bbbk} Z(\TL_n(\delta))
=
1+\left\lfloor \frac n2\right\rfloor .
\]
This dimension can also be recovered indirectly from Westbury's description
of the blocks and the Hochschild cohomology computation of de la Pe\~na and
Xi~\cite{Westbury,DeLaPenaXi}. 
In fact, we prove the following stronger filtration assertion. Let \(\T_{n,p}\) be the span
of all diagrams with at least \(p\) cups, and set \(Z_{n,p}=Z(\TL_n(\delta))\cap \T_{n,p}.\)
Then
\[
\dim_{\Bbbk}(Z_{n,p}/Z_{n,p+1})=1
\qquad
\text{for }0\le p\le \left\lfloor \frac n2\right\rfloor .
\]
The proof is by induction on \(n\). It uses only modest
representation-theoretic input, principally the known criterion for when the
radical of a Temperley--Lieb standard module is the trivial module. The
diagrammatic tools  are the operations of dilation and compression, which compare
the centers of \(\TL_n(\delta)\) and \(\TL_{n-2}(\delta)\). Dilation inserts a cup-cap pair in positions \(i\) and \(i+1\), while compression
removes one after multiplying on the left and right by a generator. Compression
sends central elements of \(\TL_n(\delta)\) to central elements of
\(\TL_{n-2}(\delta)\) and provides the inductive step between these levels. 
When the relevant radical is
the trivial module, the leading term can be written explicitly. In all other
cases, the induced compression map is injective on the central graded quotient, and the central graded
class in degree \(p\) at level \(n\) is recovered, up to scalar, from the
class in degree \(p-1\) at level \(n-2\) by inverse compression on the
associated graded quotient.

We also prove an invariance theorem for the center. The algebra
\(\TL_n(\delta)\) has a canonical anti-automorphism \(\omega\), coming from
the cellular structure, and a natural automorphism \(\sigma\), given on
generators by \(e_i\mapsto e_{n-i}\). Diagrammatically, these are reflection
through a horizontal line and reflection through a vertical line, respectively.
We prove that both maps restrict to the identity on \(Z(\TL_n(\delta))\).

In the last two sections, we give a congruence criterion for when the radical
of a standard module is the trivial module, and we record a Gram-matrix
computation for leading terms.

\section{Preliminaries}

Fix a field \(\Bbbk\) of characteristic zero and \(\delta\in\Bbbk^{\times}\). The Temperley--Lieb algebra \(\TL_n(\delta)\) is the unital associative \(\Bbbk\)-algebra generated by \(e_1,\dots,e_{n-1}\) subject to the relations
\[
e_i^2=\delta e_i,\qquad e_ie_{i\pm 1}e_i=e_i,\qquad e_ie_j=e_je_i \quad (|i-j|>1).
\]

The algebra may also be described diagrammatically. A Temperley--Lieb \(n\)-diagram consists of a planar pairing of \(2n\) marked points on a rectangle with \(n\) points on opposite edges. For example, the diagram \(d\) below is a Temperley--Lieb \(6\)-diagram. 
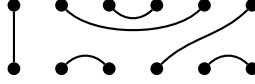
\[
\begin{tikzpicture}[scale = .42,thick, baseline={(0,-1ex/2)}] 
	\tikzstyle{vertex} = [shape = circle, minimum size = 4pt, inner sep = 1pt,
	fill=black] 
	\node[vertex] (G--6) at (7.5, -1) [shape = circle, draw] {}; 
	\node[vertex] (G--5) at (6.0, -1) [shape = circle, draw] {}; 
	\node[vertex] (G--2) at (1.5, -1) [shape = circle, draw] {}; 
	\node[vertex] (G--4) at (4.5, -1) [shape = circle, draw] {}; 
	\node[vertex] (G--3) at (3.0, -1) [shape = circle, draw] {}; 
	\node[vertex] (G--1) at (0.0, -1) [shape = circle, draw] {}; 
	\node[vertex] (G-1) at (0.0, 1) [shape = circle, draw] {}; 
	\node[vertex] (G-2) at (1.5, 1) [shape = circle, draw] {}; 
	\node[vertex] (G-3) at (3.0, 1) [shape = circle, draw] {}; 
	\node[vertex] (G-4) at (4.5, 1) [shape = circle, draw] {}; 
	\node[vertex] (G-5) at (6.0, 1) [shape = circle, draw] {}; 
	\node[vertex] (G-6) at (7.5, 1) [shape = circle, draw] {}; 
	\draw[] (G--6) .. controls +(-0.5, 0.5) and +(0.5, 0.5) .. (G--5); 
	\draw[] (G-6) .. controls +(-1, -1) and +(1, 1) .. (G--4); 
	\draw[] (G--3) .. controls +(-0.5, 0.5) and +(0.5, 0.5) .. (G--2); 
	\draw[] (G-1) .. controls +(0, -1) and +(0, 1) .. (G--1); 
	\draw[] (G-2) .. controls +(1, -1) and +(-1, -1) .. (G-5); 
	\draw[] (G-3) .. controls +(0.5, -0.5) and +(-0.5, -0.5) .. (G-4); 
\end{tikzpicture} \]

The horizontal connections in the top and bottom row are called cups and caps, respectively, and the number of each is necessarily the same in any diagram. The generator \(e_i\) corresponds to the diagram
\[
e_i \;\Leftrightarrow\;
\onepic \cdots \onepic\quad \epic \quad\onepic \cdots \onepic
\]
where the horizontal cup and cap connect the positions \(i\) and \(i+1\) on the top and bottom rows. Multiplication of diagrams is given by concatenation and removal of closed loops, each loop contributing a factor of \(\delta\). Since
this multiplication is standard, we do not review it here; see
\cite{DotyGiaquinto,RidoutSaintAubin} for details.

The algebra \(\TL_n(\delta)\) is a fundamental example of a cellular algebra in the sense of Graham and Lehrer \cite{GrahamLehrer}. We do not need the full axiomatic characterization of a cellular algebra in this paper and so only those aspects of cellular algebras necessary for our results are given as they pertain to the Temperley--Lieb algebras.

\medskip
\noindent\textbf{The cellular filtration and standard modules.}

A fundamental aspect of a cellular algebra is its natural filtration. For \(\TL_n(\delta)\) the filtration is determined by the number of cup-cap pairs present. Specifically,  let \(c(d)\) denote the number of cups in the top row of a diagram \(d\).
For each \(p\ge 0\), let \(\T_{n,p}\) be the span of all basis diagrams \(d\)
with \(c(d)\ge p\). This gives a decreasing filtration by two-sided ideals
\[
\TL_n(\delta)=\T_{n,0}\supseteq \T_{n,1}\supseteq \cdots \supseteq
\T_{n,\lfloor n/2\rfloor}\supseteq \T_{n,\lfloor n/2\rfloor+1}=0.
\]
Intersecting with the center and setting \(Z_{n,p}=Z(\TL_n(\delta)) \cap \T_{n,p}\) gives a corresponding filtration of the center
\[
Z(\TL_n(\delta))=Z_{n,0}\supseteq Z_{n,1}\supseteq \cdots \supseteq
Z_{n,\lfloor n/2\rfloor}\supseteq Z_{n,\lfloor n/2\rfloor+1}=0.
\]
If \(z\in Z_{n,p}\setminus Z_{n,p+1}\), we write \(z_{(p)}\) for its top filtration term, and we denote its class in the quotient \(Z_{n,p}/Z_{n,p+1}\) by \(\ov z=z+Z_{n,p+1}.\)

The representations of a general cellular algebra are realized axiomatically by its standard modules. For \(\TL_n(\delta)\), we will follow \cite{RidoutSaintAubin} and refer to these as standard modules. These modules have a diagrammatic description in terms of half-diagrams. 
A half-diagram on \(n\) vertices is a diagram consisting of \(n\) marked points on a horizontal line together with a collection of nonintersecting cups drawn below the line. A point not incident on a cup is called a \emph{defect}. If there are \(p\) cups, then there are \(n-2p\) defects. Let \(H_{n,p}\) denote the set of all half-diagrams with \(p\) cups and \(n-2p\) defects. For example,
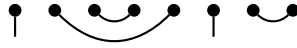
\[
\begin{tikzpicture}[scale = 0.35,thick, baseline={(0,-1ex/2)}]
	\tikzstyle{vertex} = [shape = circle, minimum size = 4pt, inner sep = 1pt, fill=black]
	
	\node[vertex] (G-1) at (0.0, 1)   [shape = circle, draw] {};
	\node[vertex] (G-2) at (1.5, 1)   [shape = circle, draw] {};
	\node[vertex] (G-3) at (3.0, 1)   [shape = circle, draw] {};
	\node[vertex] (G-4) at (4.5, 1)   [shape = circle, draw] {};
	\node[vertex] (G-5) at (6.0, 1)   [shape = circle, draw] {};
	\node[vertex] (G-6) at (7.5, 1)   [shape = circle, draw] {};
	\node[vertex] (G-7) at (9.0, 1)   [shape = circle, draw] {};
	\node[vertex] (G-8) at (10.5, 1)  [shape = circle, draw] {};
	
	\draw[] (G-1) -- (0.0,0);
	\draw[] (G-6) -- (7.5,0);
	
	\draw[] (G-2) .. controls +(1.5, -1.5) and +(-1.5, -1.5) .. (G-5);
	\draw[] (G-3) .. controls +(0.5, -0.5) and +(-0.5, -0.5) .. (G-4);
	\draw[] (G-7) .. controls +(0.5, -0.5) and +(-0.5, -0.5) .. (G-8);
	
\end{tikzpicture}
\]
is an element of \(H_{8,3}\).

For each \(p\), let \(\mathcal M_{n,p}\) denote the linear span of all half-diagrams in
\[
\bigcup_{p'\ge p}H_{n,p'}.
\]
There is a decreasing filtration
\[
\mathcal M_{n,0}\supseteq \mathcal M_{n,1}\supseteq \cdots \supseteq \mathcal M_{n,\lfloor n/2\rfloor}\supseteq \mathcal M_{n,\lfloor n/2\rfloor+1}=0
\]
and the standard module is defined to be the quotient
\[
V_{n,p}=\mathcal M_{n,p}/\mathcal M_{n,p+1}.
\]
We identify a half-diagram \(\pi\in \HH\) with its class in \(\VV\).
The action of \(\TL_n(\delta)\) on \(\VV\) is induced by a stacking procedure.
For \(d\in \TL_n(\delta)\) and \(\pi\in H_{n,p}\), stack \(d\)
on top of \(\pi\). Then
\[
d\cdot \pi=
\begin{cases}
	\delta^N\pi' & \text{if stacking produces a half-diagram \(\pi'\in \HH\),}\\
	0 & \text{if stacking produces more than \(p\) cups,}
\end{cases}
\]
where \(N\) is the number of closed loops formed in the stacking. Thus each
closed loop is removed and replaced by a factor of \(\delta\).

The one-dimensional trivial module \(T_0\) and its embeddings in the various \(\VV\) will play an important role in determining the structure of the center \(Z(\TL_n(\delta))\). Specifically, \(T_0\) is the one-dimensional \(\TL_n(\delta)\)-module which is sent to zero under the action of all \(e_i\). Thus, if \(\pi_0\) is a generator of \(T_0\) then \(e_i\cdot \pi_0=0\) for all \(i=1,\ldots, n-1\).  As an example, when \(n=3\) and \(\delta = 1\) (\( \Leftrightarrow q=e^{\pm 2\pi i/6}\)) it is easy to check that the submodule of \(V_{3,1}\) spanned by the element
\begin{equation}\label{trivialn=3}
	\setlength{\abovedisplayskip}{0pt}
	\setlength{\belowdisplayskip}{0pt}
	\begin{tikzpicture}[
		scale = 0.35, 
		thick, 
		baseline={(G-1.base)},
		vertex/.style={shape = circle, draw, minimum size = 4pt, inner sep = 1pt, fill=black}
		]
		\node[vertex] (G-1) at (0.0, 1) {};
		\node[vertex] (G-2) at (1.5, 1) {};
		\node[vertex] (G-3) at (3.0, 1) {};
		
		\draw[] (G-1) .. controls +(0.5, -0.5) and +(-0.5, -0.5) .. (G-2);
		\draw[] (G-3) -- (3.0, 0);
		
		\node[anchor=base] at (4.5, 1) {\(-\)};
		
		\node[vertex] (H-1) at (6.0, 1) {};
		\node[vertex] (H-2) at (7.5, 1) {};
		\node[vertex] (H-3) at (9.0, 1) {};
		
		\draw[] (H-1) -- (6.0, 0);
		\draw[] (H-2) .. controls +(0.5, -0.5) and +(-0.5, -0.5) .. (H-3);
	\end{tikzpicture}
\end{equation}
is an embedding of the trivial module \(T_0\) into \(V_{3,1}\). Thus as predicted by the general theory, the algebra \(\TL_3(1)\) is not semisimple as its standard module \(V_{3,1}\) has a proper submodule.

There is a canonical bilinear form 
\[
\langle\ ,\ \rangle:V_{n,p}\times V_{n,p}\to \Bbbk,
\]
on each module \(V_{n,p}\). It suffices to define the form on half-diagrams \(x,y\). To compute \(\langle x,y\rangle\), reflect \(x\) across a horizontal line,
place the result below \(y\), and connect the corresponding boundary points.
If every defect of \(x\) is connected with a defect of \(y\), then the resulting
closed diagram consists of vertical strands and some number \(r\) of closed loops, and
\[
\langle x,y\rangle=\delta^r.
\]
If the defects are not all connected, then \(\langle x,y\rangle=0\). See \cite[\S 2--\S 3]{RidoutSaintAubin} and
\cite[\S 8]{DotyGiaquinto}. 

As the bilinear form is symmetric, the left and right radicals coincide. The radical \(R_{n,p}\) of the bilinear form is defined by
\[
R_{n,p}=\{\,x\in V_{n,p}:\langle x,y\rangle=0\text{ for all }y\in V_{n,p}\,\}
\]
and is a \(\TL_n(\delta)\)-submodule of \(V_{n,p}\).

We will use the following consequences of the cellular structure of \(\TL_n(\delta)\).

\begin{lem}\label{lem:cellularfacts}
	Let \(V_{n,p}\) be the standard module for \(\TL_n(\delta)\), and let
	\(R_{n,p}\) be the radical of its bilinear form.
	
	\begin{enumerate}
		\item
		If the bilinear form on \(V_{n,p}\) is nonzero, then \(R_{n,p}\) is the unique
		maximal submodule of \(V_{n,p}\). In particular,
		\[
		L_{n,p}:=V_{n,p}/R_{n,p}
		\]
		is simple whenever it is nonzero.
		
		\item
		The nonzero modules \(L_{n,p}=V_{n,p}/R_{n,p}\), as \(p\) varies, give a
		complete set of simple \(\TL_n(\delta)\)-modules.
		
		\item
		If \(\Bbbk\) has characteristic \(0\), then \(R_{n,p}\) is either zero or
		irreducible.
		
	\item
	If \(\Bbbk\) has characteristic \(0\) and \(R_{n,p}\) is nonzero and not
	isomorphic to the trivial module \(T_0\), then \(V_{n,p}\) contains no
	submodule isomorphic to \(T_0\).
	\end{enumerate}
\end{lem}

\begin{proof}
Parts (a) and (b) are standard facts for cellular algebras. Namely, the
radical of the cellular form on a standard module is the unique maximal submodule
when the form is nonzero, the quotient by this radical is simple whenever it
is nonzero, and the nonzero such quotients give a complete set of simple
modules; see Graham--Lehrer \cite[\S 3]{GrahamLehrer}.
	
	For part (c), over \(\C\), Ridout and Saint-Aubin prove that the radical of a
	standard module is either zero or irreducible
	\cite[Theorem~7.2]{RidoutSaintAubin}. Since \(\TL_n(\delta)\), \(V_{n,p}\), and
	the bilinear form are all defined over \(\Q(\delta)\), and since the radical is
	the kernel of this form, the same irreducibility conclusion carries over to any
	field of characteristic \(0\).
	
Lastly, for (d), suppose \(T_0\) embeds in \(V_{n,p}\). Then it is a proper submodule and by
part (a), \(T_0\subseteq R_{n,p}\). It now follows from part (c) that \(T_0=R_{n,p}\), a contradiction. The proof is complete.
\end{proof}

\section{Dilations and compressions}

The inductive construction of \(Z(\TL_n(\delta))\) relies on two related maps: a dilation map from the Temperley--Lieb algebra at level \(n-2\) to level \(n\), and a compression map from level \(n\) to level \(n-2\).

Fix \(n\ge 3\) and \(1\le i\le n-1\).

	\begin{defn}
		The \(i\)th dilation map
		\[
		D_i:\TL_{n-2}(\delta)\longrightarrow \TL_n(\delta)
		\]
		is defined diagrammatically as follows. Given a diagram in
		\(\TL_{n-2}(\delta)\), insert two consecutive new nodes in positions
		\(i,i+1\) on both the top and bottom rows and shift the old nodes
		\(i,i+1,\ldots,n-2\) two positions to the right. Connect the two
		new top nodes by a cup and the two new bottom nodes by a cap, and leave all
		other connections unchanged.
\end{defn}
For example, the effect of \(D_3\) on a diagram is shown below.

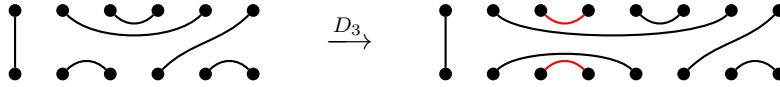
\[
\begin{tikzpicture}[scale = .42,thick, baseline={(0,-1ex/2)}] 
	\tikzstyle{vertex} = [shape = circle, minimum size = 4pt, inner sep = 1pt,
	fill=black] 
	\node[vertex] (G--6) at (7.5, -1) [shape = circle, draw] {}; 
	\node[vertex] (G--5) at (6.0, -1) [shape = circle, draw] {}; 
	\node[vertex] (G--2) at (1.5, -1) [shape = circle, draw] {}; 
	\node[vertex] (G--4) at (4.5, -1) [shape = circle, draw] {}; 
	\node[vertex] (G--3) at (3.0, -1) [shape = circle, draw] {}; 
	\node[vertex] (G--1) at (0.0, -1) [shape = circle, draw] {}; 
	\node[vertex] (G-1) at (0.0, 1) [shape = circle, draw] {}; 
	\node[vertex] (G-2) at (1.5, 1) [shape = circle, draw] {}; 
	\node[vertex] (G-3) at (3.0, 1) [shape = circle, draw] {}; 
	\node[vertex] (G-4) at (4.5, 1) [shape = circle, draw] {}; 
	\node[vertex] (G-5) at (6.0, 1) [shape = circle, draw] {}; 
	\node[vertex] (G-6) at (7.5, 1) [shape = circle, draw] {}; 
	\draw[] (G--6) .. controls +(-0.5, 0.5) and +(0.5, 0.5) .. (G--5); 
	\draw[] (G-6) .. controls +(-1, -1) and +(1, 1) .. (G--4); 
	\draw[] (G--3) .. controls +(-0.5, 0.5) and +(0.5, 0.5) .. (G--2); 
	\draw[] (G-1) .. controls +(0, -1) and +(0, 1) .. (G--1); 
	\draw[] (G-2) .. controls +(1, -1) and +(-1, -1) .. (G-5); 
	\draw[] (G-3) .. controls +(0.5, -0.5) and +(-0.5, -0.5) .. (G-4); 
\end{tikzpicture}
\qquad \xlongrightarrow{D_3} \qquad
\begin{tikzpicture}[scale = .42,thick, baseline={(0,-1ex/2)}] 
	\tikzstyle{vertex} = [shape = circle, minimum size = 4pt, inner sep = 1pt,
	fill=black] 
	\node[vertex] (G--8) at (10.5, -1) [shape = circle, draw] {}; 
	\node[vertex] (G--7) at (9.0, -1) [shape = circle, draw] {}; 
	\node[vertex] (G--6) at (7.5, -1) [shape = circle, draw] {}; 
	\node[vertex] (G-8) at (10.5, 1) [shape = circle, draw] {}; 
	\node[vertex] (G--5) at (6.0, -1) [shape = circle, draw] {}; 
	\node[vertex] (G--2) at (1.5, -1) [shape = circle, draw] {}; 
	\node[vertex] (G--4) at (4.5, -1) [shape = circle, draw] {}; 
	\node[vertex] (G--3) at (3.0, -1) [shape = circle, draw] {}; 
	\node[vertex] (G--1) at (0.0, -1) [shape = circle, draw] {}; 
	\node[vertex] (G-1) at (0.0, 1) [shape = circle, draw] {}; 
	\node[vertex] (G-2) at (1.5, 1) [shape = circle, draw] {}; 
	\node[vertex] (G-7) at (9.0, 1) [shape = circle, draw] {}; 
	\node[vertex] (G-3) at (3.0, 1) [shape = circle, draw] {}; 
	\node[vertex] (G-4) at (4.5, 1) [shape = circle, draw] {}; 
	\node[vertex] (G-5) at (6.0, 1) [shape = circle, draw] {}; 
	\node[vertex] (G-6) at (7.5, 1) [shape = circle, draw] {}; 
	\draw[] (G--8) .. controls +(-0.5, 0.5) and +(0.5, 0.5) .. (G--7); 
	\draw[] (G-8) .. controls +(-1, -1) and +(1, 1) .. (G--6); 
	\draw[] (G--5) .. controls +(-0.8, 0.8) and +(0.8, 0.8) .. (G--2); 
	\draw[red] (G--4) .. controls +(-0.5, 0.5) and +(0.5, 0.5) .. (G--3); 
	\draw[] (G-1) .. controls +(0, -1) and +(0, 1) .. (G--1); 
	\draw[] (G-2) .. controls +(1, -1) and +(-1, -1) .. (G-7); 
	\draw[red] (G-3) .. controls +(0.5, -0.5) and +(-0.5, -0.5) .. (G-4); 
	\draw[] (G-5) .. controls +(0.5, -0.5) and +(-0.5, -0.5) .. (G-6); 
\end{tikzpicture}\]
The red connections indicate the inserted cup-cap pair.

\begin{lem}\label{lem:dilationbasic}
The dilation map \(D_i\) satisfies the following properties.
\begin{enumerate}
\item \(D_i(x)e_i=e_iD_i(x)=\delta D_i(x)\) for all \(x\in \TL_{n-2}(\delta)\).
\item \(D_i(x_1)D_i(x_2)=\delta D_i(x_1x_2)\) for all \(x_1,x_2\in \TL_{n-2}(\delta)\).
\item If \(D_i'=\delta^{-1}D_i\), then \(D_i':\TL_{n-2}(\delta)\to e_i\TL_n(\delta)e_i\) is an algebra isomorphism.
\item For \(1\le j\le n-3\),
\[
D_i(e_j)=
\begin{cases}
e_ie_{j+2} & \text{if } i\le j,\\
e_{j+1}e_je_{j+2}e_{j+1} & \text{if } i=j+1,\\
e_ie_j & \text{if } i\ge j+2.
\end{cases}
\]
\end{enumerate}
\end{lem}

\begin{proof}
Each statement is a direct diagram calculation. In the products in (a) and (b), exactly one closed loop is created, and this contributes the factor \(\delta\). Part (c) is just the normalization of part (b). Part (d) is checked on generators by direct diagram calculation.
\end{proof}

\begin{rmk}\label{dilation-isomorphism}
	The dilation \(D_i\) is not an algebra isomorphism, but it is a linear isomorphism  onto \(e_i\TL_n(\delta)e_i\). By the preceding lemma, the normalized dilation
	\[
	D_i'=\delta^{-1}D_i:\TL_{n-2}(\delta)\longrightarrow e_i\TL_n(\delta)e_i
	\]
	is an algebra isomorphism, and so we can identify 
	\(e_i\TL_n(\delta)e_i\) with \(\TL_{n-2}(\delta)\). The unit of \(e_i\TL_n(\delta)e_i\) is
	\(\delta^{-1}e_i\), so this copy of \(\TL_{n-2}(\delta)\) is not a unital subalgebra of \(\TL_n(\delta)\). However, it does follow that if \(y\) is central in \(\TL_{n-2}(\delta)\) then \(D_i(y)\) is central in
	\(e_i\TL_n(\delta)e_i\).
\end{rmk}

The next lemma regarding dilation at adjacent positions is an essential tool that will be used later.

\begin{lem}\label{lem:dilationadjacent}
	For \(x\in \TL_{n-2}(\delta)\) and \(1\le i,j\le n-1\) with \(|i-j|=1\), one has
	\[
	e_jD_i(x)e_j=D_j(x).
	\]
\end{lem}

\begin{proof}
	The claim is verified by direct diagram calculation. Since both sides are linear in \(x\), it is enough to check the identity on
	the diagram basis of \(\TL_{n-2}(\delta)\).  For a basis diagram \(x\), multiplying \(D_i(x)\) on the left and right by an adjacent generator \(e_j\) moves the inserted cup-cap pair from the position
	\(i\) to the adjacent position \(j\), and creates no additional scalar factor. Thus we have the equality \(e_jD_i(x)e_j=D_j(x).\)
\end{proof}

The dilation \(D_i\) is not invertible as a map into all of
\(\TL_n(\delta)\) since it is not surjective. However, it is invertible as a map onto its image \(e_i\TL_n(\delta)e_i\). This gives a map on all of
\(\TL_n(\delta)\) by first sending \(x\) to \(e_i x e_i\) and then applying
\(D_i^{-1}\). We formalize this in the following definition.

\begin{defn}
For \(1\le i\le n-1\), the \(i\)th compression map
\[
C_i:\TL_n(\delta)\longrightarrow \TL_{n-2}(\delta)
\]
is defined by
\[
C_i(d)=D_i^{-1}(e_i d e_i)
\]
for a diagram \(d\in\TL_n(\delta)\), and then extended linearly.
\end{defn}

The following example illustrates compression.

\[\setlength{\arraycolsep}{2pt}\begin{array}{ccccc}
	&\begin{tikzpicture}[scale = .42,thick, baseline={(0,-1ex/2)}] 
		\tikzstyle{vertex} = [shape = circle, minimum size = 4pt, inner sep = 1pt,
		fill=black] 
		\node[vertex] (G--6) at (7.5, -1) [shape = circle, draw] {};  
		\node[vertex] (G--5) at (6.0, -1) [shape = circle, draw] {}; 
		\node[vertex] (G--2) at (1.5, -1) [shape = circle, draw] {}; 
		\node[vertex] (G--4) at (4.5, -1) [shape = circle, draw] {}; 
		\node[vertex] (G--3) at (3.0, -1) [shape = circle, draw] {}; 
		\node[vertex] (G--1) at (0.0, -1) [shape = circle, draw] {}; 
		\node[vertex] (G-1) at (0.0, 1) [shape = circle, draw] {}; 
		\node[vertex] (G-2) at (1.5, 1) [shape = circle, draw] {};  
		\node[vertex] (G-3) at (3.0, 1) [shape = circle, draw] {}; 
		\node[vertex] (G-4) at (4.5, 1) [shape = circle, draw] {}; 
		\node[vertex] (G-5) at (6.0, 1) [shape = circle, draw] {}; 
		\node[vertex] (G-6) at (7.5, 1) [shape = circle, draw] {}; 
		\draw[] (G--4) .. controls +(-0.5, 0.5) and +(0.5, 0.5) .. (G--3);
		\draw[] (G-3) .. controls +(0.5, -0.5) and +(-0.5, -0.5) .. (G-4); 
		\draw[] (G-1) .. controls +(0, -1) and +(0, 1) .. (G--1);
		\draw[] (G-2) .. controls +(0, -1) and +(0, 1) .. (G--2);
		\draw[] (G-5) .. controls +(0, -1) and +(0, 1) .. (G--5);
		\draw[] (G-6) .. controls +(0, -1) and +(0, 1) .. (G--6); 
	\end{tikzpicture}&&&\\
	e_3de_3\quad =& \begin{tikzpicture}[scale = .42,thick, baseline={(0,-1ex/2)}] 
		\tikzstyle{vertex} = [shape = circle, minimum size = 4pt, inner sep = 1pt,
		fill=black] 
		\node[vertex] (G--6) at (7.5, -1) [shape = circle, draw] {}; 
		\node[vertex] (G--5) at (6.0, -1) [shape = circle, draw] {}; 
		\node[vertex] (G--2) at (1.5, -1) [shape = circle, draw] {}; 
		\node[vertex] (G--4) at (4.5, -1) [shape = circle, draw] {}; 
		\node[vertex] (G--3) at (3.0, -1) [shape = circle, draw] {}; 
		\node[vertex] (G--1) at (0.0, -1) [shape = circle, draw] {}; 
		\node[vertex] (G-1) at (0.0, 1) [shape = circle, draw] {}; 
		\node[vertex] (G-2) at (1.5, 1) [shape = circle, draw] {}; 
		\node[vertex] (G-3) at (3.0, 1) [shape = circle, draw] {}; 
		\node[vertex] (G-4) at (4.5, 1) [shape = circle, draw] {}; 
		\node[vertex] (G-5) at (6.0, 1) [shape = circle, draw] {}; 
		\node[vertex] (G-6) at (7.5, 1) [shape = circle, draw] {}; 
		\draw[] (G--6) .. controls +(-0.5, 0.5) and +(0.5, 0.5) .. (G--5); 
		\draw[] (G-6) .. controls +(-1, -1) and +(1, 1) .. (G--4); 
		\draw[] (G--3) .. controls +(-0.5, 0.5) and +(0.5, 0.5) .. (G--2); 
		\draw[] (G-1) .. controls +(0, -1) and +(0, 1) .. (G--1); 
		\draw[] (G-2) .. controls +(1, -1) and +(-1, -1) .. (G-5); 
		\draw[] (G-3) .. controls +(0.5, -0.5) and +(-0.5, -0.5) .. (G-4); 
	\end{tikzpicture}&&\qquad  \xrightarrow{\scriptstyle D_3^{-1}}  \qquad
	\begin{tikzpicture}[scale = .42,thick, baseline={(0,-1ex/2)}] 
		\tikzstyle{vertex} = [shape = circle, minimum size = 4pt, inner sep = 1pt,
		fill=black] 
		\node[vertex] (G--2) at (1.5, -1) [shape = circle, draw] {}; 
		\node[vertex] (G--4) at (4.5, -1) [shape = circle, draw] {}; 
		\node[vertex] (G--3) at (3.0, -1) [shape = circle, draw] {}; 
		\node[vertex] (G--1) at (0.0, -1) [shape = circle, draw] {}; 
		\node[vertex] (G-1) at (0.0, 1) [shape = circle, draw] {}; 
		\node[vertex] (G-2) at (1.5, 1) [shape = circle, draw] {};  
		\node[vertex] (G-3) at (3.0, 1) [shape = circle, draw] {}; 
		\node[vertex] (G-4) at (4.5, 1) [shape = circle, draw] {};  
		\draw[] (G--4) .. controls +(-0.5, 0.5) and +(0.5, 0.5) .. (G--3);
		\draw[] (G-2) .. controls +(0.5, -0.5) and +(-0.5, -0.5) .. (G-3); 
		\draw[] (G-1) .. controls +(0, -1) and +(0, 1) .. (G--1);
		\draw[] (G-4) .. controls +(-1, -1) and +(1, 1) .. (G--2); 
	\end{tikzpicture} & =\quad \delta^{-1} C_3(d)
	\\
	&\begin{tikzpicture}[scale = .42,thick, baseline={(0,-1ex/2)}] 
		\tikzstyle{vertex} = [shape = circle, minimum size = 4pt, inner sep = 1pt,
		fill=black] 
		\node[vertex] (G--6) at (7.5, -1) [shape = circle, draw] {};  
		\node[vertex] (G--5) at (6.0, -1) [shape = circle, draw] {}; 
		\node[vertex] (G--2) at (1.5, -1) [shape = circle, draw] {}; 
		\node[vertex] (G--4) at (4.5, -1) [shape = circle, draw] {}; 
		\node[vertex] (G--3) at (3.0, -1) [shape = circle, draw] {}; 
		\node[vertex] (G--1) at (0.0, -1) [shape = circle, draw] {}; 
		\node[vertex] (G-1) at (0.0, 1) [shape = circle, draw] {}; 
		\node[vertex] (G-2) at (1.5, 1) [shape = circle, draw] {};  
		\node[vertex] (G-3) at (3.0, 1) [shape = circle, draw] {}; 
		\node[vertex] (G-4) at (4.5, 1) [shape = circle, draw] {}; 
		\node[vertex] (G-5) at (6.0, 1) [shape = circle, draw] {}; 
		\node[vertex] (G-6) at (7.5, 1) [shape = circle, draw] {}; 
		\draw[] (G--4) .. controls +(-0.5, 0.5) and +(0.5, 0.5) .. (G--3);
		\draw[] (G-3) .. controls +(0.5, -0.5) and +(-0.5, -0.5) .. (G-4); 
		\draw[] (G-1) .. controls +(0, -1) and +(0, 1) .. (G--1);
		\draw[] (G-2) .. controls +(0, -1) and +(0, 1) .. (G--2);
		\draw[] (G-5) .. controls +(0, -1) and +(0, 1) .. (G--5);
		\draw[] (G-6) .. controls +(0, -1) and +(0, 1) .. (G--6); 
	\end{tikzpicture}
	&&&
\end{array}\]

\begin{rmk}\label{rmk:compression-not-algebra-map}
	The compression map \(C_i\) is a linear map, but it is not an algebra map.
	The reason is that the first step \(x\mapsto e_i x e_i\) is not multiplicative.
	What will be important later is the crucial fact that compression sends
	central elements to central elements, as proved in Lemma~\ref{lem:compressioncentral}.
\end{rmk}

\begin{lem}\label{lem:compressioncentral}
If \(z\in Z(\TL_n(\delta))\), then for each \(i\) the element \(C_i(z)=D_i^{-1}(e_i z e_i)\) is central in \(\TL_{n-2}(\delta)\).
\end{lem}

\begin{proof}
Since \(z\) is central, \(e_i z e_i\) commutes with every element of the subalgebra \(e_i\TL_n(\delta)e_i\). The normalized dilation
\[
D_i'=\delta^{-1}D_i:\TL_{n-2}(\delta)\longrightarrow e_i\TL_n(\delta)e_i
\]
is an algebra isomorphism. Since
\[
e_i z e_i = D_i(C_i(z))=\delta D_i'(C_i(z)),
\]
it follows that \(\delta C_i(z)\) is central in \(\TL_{n-2}(\delta)\). Since
\(\delta\ne0\), the element \(C_i(z)\) is central.
\end{proof}

\begin{rmk}
	The maps \(D_i\) and \(C_i\) are different from some other standard
	inductive operations on Temperley--Lieb algebras. The usual inclusion obtained
	by adding vertical through-strands at the right end is not the dilation used
	here. Similarly, \(C_i\) is not the partial trace obtained
	by closing off top and bottom boundary points. 
\end{rmk}

\section{Central classes in trivial-radical case}

We treat in this section the case in which the radical of a standard module is
the trivial module. We first record how multiplication looks on a fixed
filtration piece.

Suppose \(d\) is a diagram with \(p\) cup-cap pairs. If \(d\) is cut
horizontally through the middle, it splits into a top half-diagram
\(\alpha\in H_{n,p}\) and a bottom inverted half-diagram. Reflecting the
bottom half gives another half-diagram \(\beta\in H_{n,p}\). We write \(
d=X_{\alpha,\beta}.\)
Conversely, any pair \(\alpha,\beta\in H_{n,p}\) determines such a diagram
\(X_{\alpha,\beta}\), by placing \(\alpha\) on top of a reflected \(\beta\)
and connecting defects from left to right. Thus the diagrams
\(X_{\alpha,\beta}\), with \(\alpha,\beta\in H_{n,p}\), form a basis of
\(\T_{n,p}/\T_{n,p+1}\). We extend the notation \(X_{\alpha,\beta}\) bilinearly when one or both
entries are linear combinations of half-diagrams.

In this notation, multiplication in the \(p\)th filtration quotient is given
by
\[
X_{\alpha,\beta}X_{\gamma,\eta}
\equiv
\langle \beta,\gamma\rangle X_{\alpha,\eta}
\pmod{\T_{n,p+1}}
\]
where \(\langle \beta,\gamma\rangle\) is the bilinear form on \(V_{n,p}\).

If \(y\in \T_{n,p}\), its \(p\)th filtration piece \(y_{(p)}\) is of the form
\[
y\equiv y_{(p)}
=
\sum_{\alpha,\beta\in H_{n,p}}
a_{\alpha,\beta}X_{\alpha,\beta}
\pmod{\T_{n,p+1}}
\]
for scalars \(a_{\alpha,\beta}\).

For \(\tau\in H_{n,p}\), define the lower-\(\tau\) part of \(y_{(p)}\) by
\[
y_\tau=\sum_{\alpha\in H_{n,p}}a_{\alpha,\tau}\alpha
\in V_{n,p}.
\]

\begin{lem}\label{lem:columnextraction}
	Let \(x\in \TL_n(\delta)\), let \(y\in \T_{n,p}\), and let
	\(\tau\in H_{n,p}\). Then
	\[
	(xy)_\tau=x\cdot y_\tau
	\]
	in \(V_{n,p}\).
\end{lem}

\begin{proof}
	Both sides only depend on the \(p\)th filtration piece \(y_{(p)}\), so by
	linearity it is enough to take \(y\) to be a single diagram \(X_{\alpha,\beta}\). Left multiplication by
	\(x\) changes only the upper half-diagram in the \(p\)th filtration quotient.
	The lower half remains \(\beta\). Hence the part with lower half \(\tau\) is
	zero unless \(\beta=\tau\), and in the case \(\beta=\tau\) it is exactly
	\(x\cdot\alpha\). This is precisely \(x\cdot (X_{\alpha,\beta})_\tau\).
\end{proof}

Now suppose that \(z\in Z_{n,p}\). We use the bookkeeping of the preceding
lemma to show that, in the trivial-radical case, the part of the \(p\)th filtration piece
\(z_{(p)}\) with fixed lower-\(\tau\) part must lie in the radical.

\begin{lem}\label{lem:trivialproper}
	Suppose \(p\ge 1\), let \(z\in Z_{n,p}\), and assume that \(R_{n,p}\) is the
	trivial module with generator \(\pi_0\). Then for every \(\tau_0\in H_{n,p}\),
	the submodule
	\[
	\TL_n(\delta)\cdot z_{\tau_0}
	\]
	is proper. In particular,
	\[
	z_{\tau_0}\in R_{n,p}.
	\]
\end{lem}

\begin{proof}
	Since \(z\in \T_{n,p}\) and \(p\ge 1\), every diagram appearing in \(z\) has at
	least one cup. Since \(R_{n,p}=\Bbbk\pi_0\) is trivial, every \(e_i\) kills
	\(\pi_0\). Hence every diagram with a cup kills \(\pi_0\), and so
	\[
	z\cdot \pi_0=0.
	\]
	
	Let \(y\in \TL_n(\delta)\). By Lemma~\ref{lem:columnextraction}, centrality of
	\(z\), and another use of Lemma~\ref{lem:columnextraction}, we have
	\[
	\begin{aligned}
		X_{\pi_0,\tau_0}\cdot (y\cdot z_{\tau_0})
		&=(X_{\pi_0,\tau_0}yz)_{\tau_0}\\
		&=(zX_{\pi_0,\tau_0}y)_{\tau_0}\\
		&=z\cdot (X_{\pi_0,\tau_0}y)_{\tau_0}.
	\end{aligned}
	\]
	Now \((X_{\pi_0,\tau_0}y)_{\tau_0}\) is a scalar multiple of \(\pi_0\). Indeed,
	in the \(p\)th filtration piece, the upper half is forced to be \(\pi_0\).
	Therefore the last expression is zero.
	
	Thus every element of \(\TL_n(\delta)\cdot z_{\tau_0}\) is killed by
	\(X_{\pi_0,\tau_0}\). But
	\[
	X_{\pi_0,\tau_0}\cdot \tau_0=\pi_0\ne 0.
	\]
	So \(X_{\pi_0,\tau_0}\) does not kill all of \(V_{n,p}\), and hence
	\(\TL_n(\delta)\cdot z_{\tau_0}\) is a proper submodule. Since \(R_{n,p}\) is
	the unique maximal submodule, this submodule is contained in \(R_{n,p}\).
	Thus \(z_{\tau_0}\in R_{n,p}\).
\end{proof}

There is an analogous statement for upper-half parts. Let \(\omega\) denote the
anti-involution of \(\TL_n(\delta)\) obtained by reflecting diagrams across a
horizontal axis. Thus
\[
\omega(X_{\alpha,\beta})=X_{\beta,\alpha}.
\]
Since \(\omega\) preserves the center, it follows from
Lemma~\ref{lem:trivialproper}, applied to \(\omega(z)\), that every
upper-half part of \(z_{(p)}\) also lies in \(R_{n,p}\).

\begin{thm}\label{thm:trivialshape}
	Suppose that \(R_{n,p}\) is the trivial module with generator \(\pi_0\), and let
	\[
	z\in Z_{n,p}\setminus Z_{n,p+1}.
	\]
	Then there exists \(c\in\Bbbk\) such that
	\[
	z_{(p)}=c\,X_{\pi_0,\pi_0}.
	\]
\end{thm}

\begin{proof}
	Write \(\pi_0=\sum_{\mu\in H_{n,p}}c_\mu\mu\) and
	\[
	z_{(p)}=\sum_{\tau,\mu\in H_{n,p}}a_{\tau,\mu}X_{\tau,\mu}.
	\]
	By Lemma~\ref{lem:trivialproper}, for each \(\mu\in H_{n,p}\) there is a
	scalar \(\kappa_\mu\in\Bbbk\) such that \(z_\mu=\kappa_\mu\pi_0\). Hence
	\(a_{\tau,\mu}=\kappa_\mu c_\tau\) for all \(\tau,\mu\in H_{n,p}\).
	
	Similarly, by the preceding paragraph, for each \(\tau\in H_{n,p}\) there is
	a scalar \(\kappa^\tau\in\Bbbk\) such that the upper-\(\tau\) part of \(z_{(p)}\) is \(\kappa^\tau\pi_0\). Hence
	\(a_{\tau,\mu}=\kappa^\tau c_\mu\) for all \(\tau,\mu\in H_{n,p}\).
	Therefore
	\begin{equation}\label{equalcoefficient}
		\kappa_\mu c_\tau=\kappa^\tau c_\mu
		\qquad\text{for all }\tau,\mu\in H_{n,p}.
	\end{equation}
	Pick \(\mu_0\in H_{n,p}\) such that \(c_{\mu_0}\ne 0\), and set
	\(c=\kappa_{\mu_0}/c_{\mu_0}\). Taking \(\mu=\mu_0\) in
	Equation~\eqref{equalcoefficient} gives
	\(\kappa^\tau=c\,c_\tau\) for all \(\tau\in H_{n,p}\). Thus
	\(a_{\tau,\mu}=c\,c_\tau c_\mu\) for all \(\tau,\mu\in H_{n,p}\). Hence
	\[
	z_{(p)}
	=
	\sum_{\tau,\mu\in H_{n,p}} a_{\tau,\mu}X_{\tau,\mu}
	=
	c\sum_{\tau,\mu\in H_{n,p}}c_\tau c_\mu X_{\tau,\mu}
	=
	c\,X_{\pi_0,\pi_0}.
	\]
\end{proof}

As an illustration of Theorem~\ref{thm:trivialshape} consider the algebra \(\TL_3(1)\). Then the radical \(R_{3,1}\) of standard module \(V_{3,1}\) has generator given by the diagram \(\pi_0\) appearing in Equation \eqref{trivialn=3}. In this case, \(Z_{3,2}=0\) and so there is, up to scalar, a unique central element in \(Z_{3,1}\) and it is given by

\begin{equation}\label{stacked_combination}
	\setlength{\abovedisplayskip}{0pt}
	\setlength{\belowdisplayskip}{0pt}
	X_{\pi_0,\pi_0}\quad = \quad 
	\begin{tikzpicture}[
		scale = 0.35, 
		thick, 
		baseline={(0,0)},
		vertex/.style={shape = circle, draw, minimum size = 4pt, inner sep = 1pt, fill=black}
		]
		\node[vertex] (G1-1) at (0.0, 1) {};
		\node[vertex] (G1-2) at (1.5, 1) {};
		\node[vertex] (G1-3) at (3.0, 1) {};
		\node[vertex] (G1-1b) at (0.0, -1) {};
		\node[vertex] (G1-2b) at (1.5, -1) {};
		\node[vertex] (G1-3b) at (3.0, -1) {};
		
		\draw[] (G1-1) .. controls +(0.5, -0.5) and +(-0.5, -0.5) .. (G1-2);
		\draw[] (G1-1b) .. controls +(0.5, 0.5) and +(-0.5, 0.5) .. (G1-2b);
		\draw[] (G1-3) -- (G1-3b);
		
		\node at (4.25, 0) {\(-\)};
		
		\node[vertex] (G2-1) at (5.5, 1) {};
		\node[vertex] (G2-2) at (7.0, 1) {};
		\node[vertex] (G2-3) at (8.5, 1) {};
		\node[vertex] (G2-1b) at (5.5, -1) {};
		\node[vertex] (G2-2b) at (7.0, -1) {};
		\node[vertex] (G2-3b) at (8.5, -1) {};
		
		\draw[] (G2-1) .. controls +(0.5, -0.5) and +(-0.5, -0.5) .. (G2-2);
		\draw[] (G2-2b) .. controls +(0.5, 0.5) and +(-0.5, 0.5) .. (G2-3b);
		\draw[] (G2-3) .. controls +(-0.5, -0.5) and +(0.5, 0.5) .. (G2-1b);
		
		\node at (9.75, 0) {\(-\)};
		
		\node[vertex] (G3-1) at (11.0, 1) {};
		\node[vertex] (G3-2) at (12.5, 1) {};
		\node[vertex] (G3-3) at (14.0, 1) {};
		\node[vertex] (G3-1b) at (11.0, -1) {};
		\node[vertex] (G3-2b) at (12.5, -1) {};
		\node[vertex] (G3-3b) at (14.0, -1) {};
		
		\draw[] (G3-2) .. controls +(0.5, -0.5) and +(-0.5, -0.5) .. (G3-3);
		\draw[] (G3-1b) .. controls +(0.5, 0.5) and +(-0.5, 0.5) .. (G3-2b);
		\draw[] (G3-1) .. controls +(0.5, -0.5) and +(-0.5, 0.5) .. (G3-3b);
		
		\node at (15.25, 0) {\(+\)};
		
		\node[vertex] (G4-1) at (16.5, 1) {};
		\node[vertex] (G4-2) at (18.0, 1) {};
		\node[vertex] (G4-3) at (19.5, 1) {};
		\node[vertex] (G4-1b) at (16.5, -1) {};
		\node[vertex] (G4-2b) at (18.0, -1) {};
		\node[vertex] (G4-3b) at (19.5, -1) {};
		
		\draw[] (G4-2) .. controls +(0.5, -0.5) and +(-0.5, -0.5) .. (G4-3);
		\draw[] (G4-2b) .. controls +(0.5, 0.5) and +(-0.5, 0.5) .. (G4-3b);
		\draw[] (G4-1) -- (G4-1b);
	\end{tikzpicture}
\end{equation}

\section{A lower bound by deformation}

In this section we provide a lower bound for the dimension of the center by an
elementary argument coming from deformation theory.

\begin{prop}\label{prop:lowerbound}
	Let \(\Bbbk\) be a field of characteristic zero and let
	\(\delta\in\Bbbk^\times\). Then
	\[
	\dim_{\Bbbk} Z(\TL_n(\delta))
	\ge
	1+\left\lfloor \frac n2\right\rfloor .
	\]
\end{prop}

\begin{proof}
	Let \(t\) be a formal variable and consider the \(\Bbbk[t]\)-algebra
	\[
	A_t=\TL_{n,\Bbbk[t]}(\delta+t),
	\]
	which has the usual diagram basis and multiplication rule
	\[
	d*d'=(\delta+t)^N dd',
	\]
	where \(N\) is the number of loops removed and \(dd'\) is the resulting part after loop removal. Setting \(t=0\)
	gives the \(\Bbbk\)-algebra
	\[
	A_0=\TL_n(\delta).
	\]
	As a \(\Bbbk[t]\)-module, \(A_t\) is free with basis the TL diagrams. The
	center \(Z(A_t)\) is also free, since it is a submodule of a free module over
	the PID \(\Bbbk[t]\). Let \(r\) be the rank of \(Z(A_t)\).
	
	Specializing to \(t=0\) preserves centrality, and so there is a map
	\[
	s:Z(A_t)\longrightarrow Z(A_0).
	\]
	If \(z_t\in\ker(s)\), then all diagram coefficients of \(z_t\) are divisible
	by \(t\), so \(z_t=tw_t\) for some \(w_t\in A_t\). Since \(z_t\) is central,
	\(w_t\) is also central since \(0=[z_t,a]=t[w_t,a]\) for all \(a\in A_t\) and \(A_t\) has no \(t\)-torsion. Thus \(\ker(s)=tZ(A_t).\) 
	Therefore the specialization to \(t=0\) induces an injection
	\[
	Z(A_t)/tZ(A_t)\hookrightarrow Z(A_0).
	\]
	Since \(Z(A_t)\) is free of rank \(r\) over \(\Bbbk[t]\), the quotient
	\(Z(A_t)/tZ(A_t)\) has \(\Bbbk\)-dimension \(r\),
	and so \(r\leq \dim_{\Bbbk}Z(A_0).\)
	
	Now extend scalars to the field \(K=\Bbbk(t)\). Then
	\(\TL_{n,K}(\delta+t)\) is split semisimple, since \(\delta+t\) is generic, and
	so its center has dimension \(1+\lfloor n/2\rfloor\). Moreover, after scalar extension the rank \(r\)
	of \(Z(A_t)\) becomes the dimension of the center of
	\(\TL_{n,K}(\delta+t)\). Therefore
	\[
	1+\left\lfloor \frac n2\right\rfloor
	=
	r
	\leq
	\dim_{\Bbbk}Z(A_0)
	=
	\dim_{\Bbbk}Z(\TL_n(\delta)).
	\]
	This proves the result.
\end{proof}

\section{An upper bound for the dimension of the center}

We now prove an upper bound for \(\dim Z(\TL_n(\delta))\). This upper bound will coincide with the lower bound from the previous section and hence it determines the dimension of the center. 

\begin{lem}\label{lem:pzero}
	For every \(n\),
	\[
	\dim (Z_{n,0}/Z_{n,1})= 1.
	\]
\end{lem}

\begin{proof}
	The only basis diagram with no cups is the identity diagram, and this is central. Consequently, \(Z_{n,0}/Z_{n,1}\) has dimension one for all \(n\). 
\end{proof}

\begin{lem}\label{lem:quotientcompression}
	Assume \(n\ge 3\). For each \(i=1,\dots,n-1\) and each \(p\ge 1\), compression induces a linear map
	\[
	C_i: Z_{n,p}/Z_{n,p+1}\longrightarrow Z_{n-2,p-1}/Z_{n-2,p}.
	\]
\end{lem}

\begin{proof}
	If \(z\in Z_{n,p}\), then every diagram appearing in \(z\) has at least \(p\)
	cups, so every diagram appearing in \(C_i(z)\) has at least \(p-1\) cups. By
	Lemma~\ref{lem:compressioncentral}, \(C_i(z)\) is central in
	\(\TL_{n-2}(\delta)\), hence lies in \(Z_{n-2,p-1}\). Similarly, if
	\(z\in Z_{n,p+1}\), then \(C_i(z)\in Z_{n-2,p}\). Thus compression descends to
	the quotient.
\end{proof}

\begin{lem}\label{lem:trivialdimension}
	Suppose that \(p\ge 1\) and that \(R_{n,p}\) is the trivial module. Then
	\[
	\dim (Z_{n,p}/Z_{n,p+1})\leq 1.
	\]
\end{lem}

\begin{proof}
	By Theorem~\ref{thm:trivialshape}, every nonzero class in
	\(Z_{n,p}/Z_{n,p+1}\) is represented by a scalar multiple of
	\(X_{\pi_0,\pi_0}\). Thus the quotient is at most one-dimensional.
\end{proof}

The next lemma shows that whenever \(V_{n,p}\) does not contain a copy of the trivial module, then every
nonzero class in \(Z_{n,p}/Z_{n,p+1}\) has at least one compression which is not zero.

\begin{lem}\label{lem:detection}
	Assume \(n\ge 3\) and \(p\ge 1\) and suppose that \(V_{n,p}\) contains no nonzero trivial
	submodule. Let \(z\in Z_{n,p}\setminus Z_{n,p+1}\). Then there exists
	\(i\in\{1,\dots,n-1\}\) such that
	\[
	C_i(z)\notin Z_{n-2,p}.
	\]
	Equivalently, if \(\overline z=z+Z_{n,p+1}\), then
	\[
	C_i(\overline z)\ne 0
	\]
	in \(Z_{n-2,p-1}/Z_{n-2,p}\).
\end{lem}

\begin{proof}
	Suppose the claim is not true. Then
	\[
	C_i(z)\in Z_{n-2,p}
	\qquad\text{for all}\qquad i=1,\ldots,n-1.
	\]
   For each \(\mu\in H_{n,p}\), let \(z_\mu\) be the lower-\(\mu\) part of
   \(z_{(p)}\). We claim
   \[
   e_i\cdot z_\mu=0
   \qquad\text{for all } i=1,\ldots,n-1.
   \]
	Indeed, left multiplication by \(e_i\) acts on the upper half-diagram in
	\(z_{(p)}\) and leaves the lower half fixed. The element \(C_i(z)\) is obtained
	from \(e_i z e_i\) by removing the local cup-cap pair at positions \(i,i+1\).
	Consequently, the filtration-\((p-1)\) part of \(C_i(z)\) records exactly the
	vectors \(e_i\cdot z_\mu\), as \(\mu\) ranges over \(H_{n,p}\).  Since, by
	assumption, \(C_i(z)\in Z_{n-2,p}\), this filtration-\((p-1)\) part is zero.
	Thus \(e_i\cdot z_\mu=0\).
	
	Since \(z\notin Z_{n,p+1}\), its leading term \(z_{(p)}\) is nonzero, so
	there exists \(\mu_0\in H_{n,p}\) such that \(z_{\mu_0}\ne 0\). Then every
	generator \(e_i\) kills \(z_{\mu_0}\), and hence the submodule of
	\(V_{n,p}\) generated by \(z_{\mu_0}\) is a nonzero trivial submodule. This
	contradicts the hypothesis. Therefore some compression is nonzero in
	\(Z_{n-2,p-1}/Z_{n-2,p}\).
\end{proof}

The next lemma shows that, once the target quotient is one-dimensional, the
compression scalars do not depend on \(i\).

\begin{lem}\label{lem:equalscalars}
	Assume \(n\ge3\), \(p\ge1\), and
	\[
	\dim\left(Z_{n-2,p-1}/Z_{n-2,p}\right)=1.
	\]
	Choose \(y\in Z_{n-2,p-1}\) whose image in the quotient is nonzero.
	For \(\overline z\in Z_{n,p}/Z_{n,p+1}\), write
	\[
	C_i(\overline z)=\lambda_{i,z}(y+Z_{n-2,p}).
	\]
	Then the scalars \(\lambda_{i,z}\) are independent of \(i\).
\end{lem}

\begin{proof}
	The assertion is clear if \(\overline z=0\). Otherwise, let
	\(z\in Z_{n,p}\setminus Z_{n,p+1}\) represent \(\overline z\). Since
	\[
	C_i(\overline z)=\lambda_{i,z}(y+Z_{n-2,p}),
	\]
	we have
	\[
	e_i z e_i\equiv \lambda_{i,z}D_i(y)
	\pmod{\T_{n,p+1}}.
	\]
	Fix \(i=1,\ldots,n-2\). Multiplying this congruence on the left by
	\(e_{i+1}\), and using centrality of \(z\), gives
	\[
	\lambda_{i,z}e_{i+1}D_i(y)
	\equiv
	e_{i+1}e_i z e_i
	=
	e_{i+1}z e_i^2
	=
	\delta e_{i+1}z e_i
	=
	\delta e_{i+1}e_i z
	\pmod{\T_{n,p+1}}.
	\]
	Similarly, from
	\[
	e_{i+1} z e_{i+1}\equiv \lambda_{i+1,z}D_{i+1}(y)
	\pmod{\T_{n,p+1}},
	\]
	multiplying on the right by \(e_i\) gives
	\[
	\begin{aligned}
	\lambda_{i+1,z}D_{i+1}(y)e_i
	&\equiv e_{i+1}z e_{i+1}e_i \\
	&= z e_{i+1}^2 e_i \\
	&= \delta z e_{i+1}e_i \\
	&= \delta e_{i+1}e_i z
	\pmod{\T_{n,p+1}}.
	\end{aligned}
	\]
	Therefore
	\[
	\lambda_{i,z}e_{i+1}D_i(y)
	\equiv
	\lambda_{i+1,z}D_{i+1}(y)e_i
	\pmod{\T_{n,p+1}}.
	\]
	By Lemma~\ref{lem:dilationadjacent},
	\[
	D_{i+1}(y)=e_{i+1}D_i(y)e_{i+1}.
	\]
	Hence
	\[
	\lambda_{i,z}e_{i+1}D_i(y)
	\equiv
	\lambda_{i+1,z}e_{i+1}D_i(y)e_{i+1}e_i
	\pmod{\T_{n,p+1}}.
	\]
	Since \(D_i(y)e_i=\delta D_i(y)\), the right-hand side is
	\[
	\begin{aligned}
		\lambda_{i+1,z}e_{i+1}D_i(y)e_{i+1}e_i
		&=
		\delta^{-1}\lambda_{i+1,z}e_{i+1}D_i(y)e_i e_{i+1}e_i  \\
		&=
		\delta^{-1}\lambda_{i+1,z}e_{i+1}D_i(y)e_i  \\
		&=
		\lambda_{i+1,z}e_{i+1}D_i(y).
	\end{aligned}
	\]
	Thus
	\[
	(\lambda_{i,z}-\lambda_{i+1,z})e_{i+1}D_i(y)
	\equiv 0
	\pmod{\T_{n,p+1}}.
	\]
	Multiplying on the left by \(e_i\), and using
	\[
	e_i e_{i+1}D_i(y)
	=
	\delta^{-1}e_i e_{i+1}e_iD_i(y)
	=
	\delta^{-1}e_iD_i(y)
	=
	D_i(y),
	\]
	we get
	\[
	(\lambda_{i,z}-\lambda_{i+1,z})D_i(y)
	\equiv 0
	\pmod{\T_{n,p+1}}.
	\]
	Since \(y\notin Z_{n-2,p}\), the element \(D_i(y)\) is nonzero modulo
	\(\T_{n,p+1}\). Hence
	\[
	\lambda_{i,z}=\lambda_{i+1,z}.
	\]
	This holds for every \(i=1,\ldots,n-2\), so all the scalars are equal.
\end{proof}

\begin{lem}\label{lem:nontrivialdimension}
 Assume \(n\ge 3\) and \(p\ge 1\), and suppose that \(V_{n,p}\) contains no nonzero trivial
submodule. Then
\[
\dim(Z_{n-2,p-1}/Z_{n-2,p})\le 1 \quad \Longrightarrow \quad
\dim(Z_{n,p}/Z_{n,p+1})\le 1.
\]
\end{lem}

\begin{proof}
	Assume
	\[
	\dim(Z_{n-2,p-1}/Z_{n-2,p})\le 1.
	\]
	If \(Z_{n,p}/Z_{n,p+1}=0\), there is nothing to prove. Thus choose
	\[
	0\ne \overline z\in Z_{n,p}/Z_{n,p+1}.
	\]
	By Lemma~\ref{lem:detection}, some compression of \(\overline z\) is
	nonzero in \(Z_{n-2,p-1}/Z_{n-2,p}\). Therefore
	\[
	\dim (Z_{n-2,p-1}/Z_{n-2,p})=1.
	\]
	Choose \(y\in Z_{n-2,p-1}\) whose image in
	\(Z_{n-2,p-1}/Z_{n-2,p}\) is nonzero.
	
	Fix \(i\). We claim that the induced compression map
	\[
	C_i: Z_{n,p}/Z_{n,p+1}\longrightarrow Z_{n-2,p-1}/Z_{n-2,p},
	\qquad
	\overline w\longmapsto C_i(\overline w)
	\]
	is injective. Indeed, suppose \(C_i(\overline w)=0\). If
	\(\overline w\ne 0\), then Lemma~\ref{lem:equalscalars} implies that all
	compressions of \(\overline w\) vanish, contradicting
	Lemma~\ref{lem:detection}. Hence \(\overline w=0\), and so \(C_i\) is
	injective.
	
	Therefore
	\[
	\dim (Z_{n,p}/Z_{n,p+1})
	\le
	\dim (Z_{n-2,p-1}/Z_{n-2,p})
	=1.
	\]
\end{proof}

\begin{prop}\label{thm:upperbound}
	Fix \(n\ge 1\), let \(\Bbbk\) be a field of characteristic zero, and let
	\(\delta\in\Bbbk^{\times}\).
	\begin{enumerate}
		\item For all \(0\le p\le \lfloor n/2\rfloor\),
		\[
		\dim (Z_{n,p}/Z_{n,p+1})\le 1.
		\]
		\item One has
		\[
		\dim Z(\TL_n(\delta))\le 1+\Bigl\lfloor \frac{n}{2}\Bigr\rfloor.
		\]
	\end{enumerate}
\end{prop}

\begin{proof}
The proof is by induction on \(n\). For \(n=1,2\), the algebra
\(\TL_n(\delta)\) is commutative, so the statement is clear.

Assume \(n\ge 3\), and assume the result is known at level \(n-2\).
Fix \(p\) with \(0\le p\le \lfloor n/2\rfloor\). If \(p=0\), then
part (a) follows from Lemma~\ref{lem:pzero}.

Now assume \(p\ge 1\). By the induction hypothesis,
\[
\dim(Z_{n-2,p-1}/Z_{n-2,p})\le 1.
\]
If \(R_{n,p}\) is the trivial module, then
part (a) follows from Lemma~\ref{lem:trivialdimension}.

If \(R_{n,p}\) is nonzero and not the trivial module, then
Lemma~\ref{lem:cellularfacts} implies that \(V_{n,p}\) contains no nonzero
trivial submodule, and part (a) follows from Lemma~\ref{lem:nontrivialdimension}.

If \(R_{n,p}=0\), then \(V_{n,p}\) is simple. Since \(n\ge 3\) and
\(p\ge 1\), we have \(|H_{n,p}|>1\), so \(V_{n,p}\) is not the
one-dimensional trivial module. Hence it contains no nonzero trivial
submodule, and again part (a) follows from Lemma~\ref{lem:nontrivialdimension}.

Thus every graded piece has dimension at most one.

Summing over all \(p\) gives
\[
\dim Z(\TL_n(\delta))
\le
\sum_{p=0}^{\lfloor n/2\rfloor} \dim (Z_{n,p}/Z_{n,p+1})
\le 1+\Bigl\lfloor \frac{n}{2}\Bigr\rfloor.
\]
\end{proof}

\begin{thm}\label{thm:main}
	Let \(\Bbbk\) be a field of characteristic zero and let
	\(\delta\in\Bbbk^\times\). Then:
	\begin{enumerate}
		\item For every \(0\le p\le \lfloor n/2\rfloor\),
		\[
		\dim(Z_{n,p}/Z_{n,p+1})=1.
		\]
		\item The center has dimension
		\[
		\dim Z(\TL_n(\delta))
		=
		1+\left\lfloor \frac n2\right\rfloor .
		\]
	\end{enumerate}
\end{thm}

\begin{proof}
	By Proposition~\ref{thm:upperbound}, each quotient
	\(Z_{n,p}/Z_{n,p+1}\) has dimension at most one. Hence
	\[
	\dim Z(\TL_n(\delta))
	=
	\sum_{p=0}^{\lfloor n/2\rfloor}
	\dim_{\Bbbk}(Z_{n,p}/Z_{n,p+1})
	\le
	1+\left\lfloor \frac n2\right\rfloor .
	\]
	The reverse inequality is Proposition~\ref{prop:lowerbound}. Therefore equality
	holds in the displayed inequality. Since there are exactly
	\(1+\lfloor n/2\rfloor\) summands, and each summand is at most one, each summand
	must be equal to one. This proves both parts.
\end{proof}

\section{Invariance of the center}

There is an action of the group \(C_2\times C_2\) on \(\TL_n(\delta)\) induced by a pair of commuting involutions. The algebra anti-automorphism
\[
\omega:\TL_n(\delta)\to \TL_n(\delta)
\]
and algebra automorphism
\[
\sigma:\TL_n(\delta)\to \TL_n(\delta)
\]
are defined on products of generators by
\[
\omega(e_{i_1}\cdots e_{i_s})=e_{i_s}\cdots e_{i_1},
\qquad
\sigma(e_{i_1}\cdots e_{i_s})=e_{n-i_1}\cdots e_{n-i_s},
\]
and extended linearly. Diagrammatically, \(\omega\) is reflection across the horizontal bisector and \(\sigma\) is reflection across the vertical bisector.

\begin{lem}\label{lem:reflectionscompression}
The involutions \(\omega\) and \(\sigma\) preserve the filtration \(Z_{n,p}\). Moreover, for each \(i=1,\dots,n-1\) and each central element \(z\in Z_{n,p}\),
\[
\omega(C_i(z))=C_i(\omega(z))
\qquad \text{and} \qquad
\sigma(C_i(z))=C_{n-i}(\sigma(z)).
\]
Consequently, both \(\omega\) and \(\sigma\) induce linear maps on each quotient \(Z_{n,p}/Z_{n,p+1}\).
\end{lem}

\begin{proof}
Both \(\omega\) and \(\sigma\) preserve the diagram basis and the number of cups, hence preserve each \(\T_{n,p}\). They also preserve the center, so they preserve \(Z_{n,p}=Z(\TL_n(\delta))\cap \T_{n,p}\).

For compatibility with compression, note that \(\omega(e_i)=e_i\) for all \(i\), and diagrammatically \(\omega\) commutes with removing the local cup-cap pair at positions \(i,i+1\). Thus
\[
\omega(C_i(z))=\omega(D_i^{-1}(e_i z e_i))=D_i^{-1}(e_i\omega(z)e_i)=C_i(\omega(z)).
\]
Similarly, \(\sigma(e_i)=e_{n-i}\), and \(\sigma\) sends the local cup-cap pair at positions \(i,i+1\) to the one at positions \(n-i,n-i+1\). Hence
\[
\sigma(C_i(z))=C_{n-i}(\sigma(z)).
\]
The induced maps on the quotients are immediate.
\end{proof}

\begin{thm}\label{thm:invariance}
Fix \(n\ge 1\), let \(\Bbbk\) be a field of characteristic zero, and let \(\delta\in\Bbbk^{\times}\).
\begin{enumerate}
\item For all \(p=0,\dots,\lfloor n/2\rfloor\), the induced maps of \(\omega\) and \(\sigma\) on \(Z_{n,p}/Z_{n,p+1}\) are
the identity.
\item Every element of \(Z(\TL_n(\delta))\) is fixed by \(\omega\) and \(\sigma\).
\end{enumerate}
\end{thm}

\begin{proof}
We first prove part (a). If \(p=0\), then \(Z_{n,0}/Z_{n,1}\) is spanned by the identity class, and the
claim is clear. We now proceed under the assumption that \(p\geq 1\). If \(R_{n,p}\) is the trivial module with generator \(\pi_0\), then Theorem~\ref{thm:trivialshape} shows that \(Z_{n,p}/Z_{n,p+1}\) is spanned by \(X_{\pi_0,\pi_0}\). It is immediate that \(\omega(X_{\pi_0,\pi_0})=X_{\pi_0,\pi_0}\). Also \(\sigma(\pi_0)=\pm \pi_0\), so \(\sigma(X_{\pi_0,\pi_0})=X_{\pi_0,\pi_0}\).

In this case \(V_{n,p}\) contains no nonzero trivial submodule: if
\(R_{n,p}\ne0\), this follows from Lemma~\ref{lem:cellularfacts}; if
\(R_{n,p}=0\), then \(V_{n,p}\) is simple and, since \(n\ge3\) and \(p\ge1\),
we have \(|H_{n,p}|>1\), so \(V_{n,p}\) is not the one-dimensional trivial
module. 

Let \(0\neq \ov z\in Z_{n,p}/Z_{n,p+1}\). By Lemma~\ref{lem:detection}, there exists \(i\) such that \(C_i(\ov z)\neq 0\). Since \(Z_{n-2,p-1}/Z_{n-2,p}\) is one-dimensional and fixed by induction, both \(\omega\) and \(\sigma\) act trivially on \(C_i(\ov z)\).

Since \(Z_{n,p}/Z_{n,p+1}\) is one-dimensional and \(\omega\) and \(\sigma\)
are involutions, each induced map acts as multiplication by \(1\) or \(-1\).
If \(\omega(\ov z)=-\ov z\), then Lemma~\ref{lem:reflectionscompression} gives
\[
\omega(C_i(\ov z))=C_i(\omega(\ov z))=-C_i(\ov z),
\]
contradicting the induction hypothesis. Hence \(\omega(\ov z)=\ov z\). 

Now suppose that \(\sigma(\overline z)=-\overline z\). By
Lemma~\ref{lem:reflectionscompression},
\[
\sigma(C_i(\overline z))
=
C_{n-i}(\sigma(\overline z))
=
-C_{n-i}(\overline z).
\]
By Lemma~\ref{lem:equalscalars}, all compressions of \(\overline z\) are the
same scalar multiple of a fixed generator of
\(Z_{n-2,p-1}/Z_{n-2,p}\). Since \(C_i(\overline z)\ne0\), it follows that
\( C_{n-i}(\overline z)=C_i(\overline z)\ne 0\), and so
\(\sigma(C_i(\overline z))=-C_i(\overline z)\).
This contradicts the induction hypothesis, since \(\sigma\) acts trivially on
\(Z_{n-2,p-1}/Z_{n-2,p}\). Therefore
\(\sigma(\overline z)=\overline z\).

For part (b), let \(z\in Z(\TL_n(\delta))\). Choose \(p\) with \(z\in Z_{n,p}\setminus Z_{n,p+1}\). Since \(\omega\) acts trivially on \(Z_{n,p}/Z_{n,p+1}\), one has \(\omega(z)-z\in Z_{n,p+1}\). Applying the same argument repeatedly down the finite filtration shows that \(\omega(z)-z\) lies in every \(Z_{n,r}\) and hence is zero. Thus \(\omega(z)=z\). The same argument applies to \(\sigma\).
\end{proof}

\section{Congruence criterion for trivial radicals}

The goal of this section is to establish a congruence condition which determines precisely when \(R_{n,p}\) is the
trivial module.  This question can be analyzed using the block structure of
\(\TL_n(\delta)\) at roots of unity, studied by Goodman and Wenzl in
\cite{GoodmanWenzl}.  For our purposes, it is more convenient to use the related
terminology and results of \cite[Section 5]{RidoutSaintAubin}, which we
recall as needed.  The main point is that there is a ``reflection formula''
which determines the dimension of \(R_{n,p}\) in terms of a reflected simple
module \(L_{n,p'}\), where \(p'<p\).  We will see that the only way this
reflected simple module can have dimension one is when \(p'=0\).  This will
give the desired congruence condition.

Assume that \(\delta=q+q^{-1}\ne 0\).  Following \cite{RidoutSaintAubin}, the pair \((n,p)\) is
called critical if
\(q^{2(n-2p+1)}=1.\)
Suppose that \(q\) is a root of unity, and let \(\ell\) be the smallest
positive integer such that
\(q^{2\ell}=1.\)
Then \((n,p)\) is critical if and only if
\(\ell\mid n-2p+1.\)
Equivalently, if
\[
n-2p+1=k\ell+r,\qquad 1\le r\le \ell
\]
for some \(k\in \mathbb Z\), then \((n,p)\) is critical exactly when \(r=\ell\).

\begin{lem} If \(L_{n,p}=V_{n,p}/R_{n,p}\) is nonzero, then
	\(L_{n,p}\) is the trivial module if and only if \(p=0\).\label{lem:trivialquotient}
\end{lem}
\begin{proof}
	
	The module \(V_{n,0}\) is one-dimensional and spanned by the diagram with all defects. Every \(e_i\) acts by zero on this diagram, so \(V_{n,0}=L_{n,0}\) is trivial. 
	
	Now suppose \(p>0\), and assume that \(L_{n,p}\) is trivial. Every half-diagram
	\(\pi\in H_{n,p}\) has at least one cup connecting two adjacent vertices
	\(i\) and \(i+1\). For this \(i\), we have \(
	e_i\cdot \pi=\delta \pi.\)
	If \(\ov{\pi}\) is the class of \(\pi\) in \(L_{n,p}\), then \(e_i\cdot \ov{\pi}=0\). On the other hand, \(e_i\cdot \ov{\pi}=\delta  \ov{\pi}\). Since \(\delta\neq 0\), the class \(\ov{\pi}=0\) for all \(\pi\in H_{n,p}\) and so \(L_{n,p}=0\), a contradiction. Thus if \(p>0\), no nonzero \(L_{n,p}\) is the trivial module. 
\end{proof}

\begin{thm}\label{thm:trivialradicalcriterion}
	Assume that \(\delta=q+q^{-1}\ne 0\), and suppose that \(q\) is a root of
	unity. Let \(\ell\) be the smallest positive integer such that
	\[
	q^{2\ell}=1.
	\]
	Then
	\[
	R_{n,p}\text{ is the trivial module}
	\quad\Longleftrightarrow\quad
	1\le p\le \ell-1
	\quad\text{and}\quad
	\ell\mid n-p+1.
	\]
\end{thm}

\begin{proof}
	Write
	\[
	n-2p+1=k\ell+r,
	\qquad 1\le r\le \ell.
	\]
	Then \((n,p)\) is critical exactly when \(r=\ell\). In this case
	\(R_{n,p}=0\) by \cite[Proposition~5.1]{RidoutSaintAubin}, so the radical is
	not the trivial module. Thus we may assume \(1\le r\le \ell-1\).
	
	Put \(p'=p+r-\ell.\) By \cite[Proposition~5.3]{RidoutSaintAubin}, if \(p'<0\), then
	\(R_{n,p}=0\). We may therefore assume \(p'\ge 0\). In this case,
	\cite[Theorem~7.2]{RidoutSaintAubin} gives
	\[
	R_{n,p}\cong L_{n,p'}.\]
	
	Therefore \(R_{n,p}\) is trivial exactly when \(L_{n,p'}\) is trivial. By
	Lemma~\ref{lem:trivialquotient}, this happens exactly when \(p'=0\). So the radical
	\(R_{n,p}\) is trivial if and only if \(r=\ell -p\).
	
	Since \(1\le r\le \ell-1\), the equation \(r=\ell-p\) is possible exactly when
	\[
	1\le \ell-p\le \ell-1,
	\]
	or equivalently
	\[
	1\le p\le \ell-1.
	\]
	Also, because \(r\equiv n-2p+1\pmod{\ell}\), the condition \(r=\ell-p\) is
	equivalent to
	\[
	n-2p+1\equiv -p\pmod{\ell},
	\]
	that is,
	\[
	\ell\mid n-p+1.
	\]
	This proves the theorem.
\end{proof}

\begin{ex}\label{ex:deltaonepattern}
	Suppose \(\delta=1\). Then \(\ell=3\), and
	Theorem~\ref{thm:trivialradicalcriterion} says that \(R_{n,p}\) is the
	trivial module exactly when
	\[
	1\le p\le 2
	\qquad\text{and}\qquad
	3\mid n-p+1.
	\]
	Thus
	\[
	R_{n,1}\text{ is trivial}\quad\Longleftrightarrow\quad n\equiv 0\pmod 3,
	\]
	and
	\[
	R_{n,2}\text{ is trivial}\quad\Longleftrightarrow\quad n\equiv 1\pmod 3.
	\]
	In these cases, the embedding of the trivial module in $V_{n,p}$ is easy to describe. There is no trivial-radical case with \(p\ge 3\).
	
	Now suppose \(n\equiv 2\pmod 3\). Then neither \(R_{n,1}\) nor \(R_{n,2}\)
	is trivial. Thus the first possible nonzero center pieces after the identity
	piece are not of the trivial-radical type. They are instead controlled by
	compression: any nonzero class in
	\[
	Z_{n,p}/Z_{n,p+1}
	\]
	is detected by some compression to
	\[
	Z_{n-2,p-1}/Z_{n-2,p},
	\]
	and its nonzero compression is a scalar multiple of the unique central class
	two nodes down.
\end{ex}

	\section{A Gram matrix computation of the leading term}
The preceding sections prove the one-dimensionality of
\(Z_{n,p}/Z_{n,p+1}\). This final section is a computational supplement. It
records the leading-term equations and explains how the Gram matrix of the
bilinear form on \(\VV\) can be used to solve them. It does not give an
independent proof of the dimension statement.

Write
\[
z_{(p)}=\sum_{\alpha,\beta\in \HH}c_{\alpha,\beta}X_{\alpha,\beta},
\]
where \(c_{\alpha,\beta}\in \Bbbk\), and let
\[
C=(c_{\alpha,\beta})_{\alpha,\beta\in \HH}.
\]
Fix \(i\), set \(e=e_i\), and suppose that the action of \(e\) on \(\VV\) is
given by
\[
e\cdot \alpha=\sum_{\gamma\in \HH}a_{\gamma,\alpha}\gamma.
\]
Let
\[
A=(a_{\gamma,\alpha})_{\gamma,\alpha\in\HH}.
\]
Thus \(A\) is the matrix describing the action of \(e\) on \(\VV\), with
respect to the basis \(\HH\).

Modulo \(\T_{n,p+1}\), left multiplication by \(e\) acts only on the top half
diagram. Hence
\[
eX_{\alpha,\beta}
\equiv
\sum_{\gamma\in \HH}a_{\gamma,\alpha}X_{\gamma,\beta}
\pmod{\T_{n,p+1}}.
\]
Therefore
\[
ez_{(p)}
\equiv
\sum_{\alpha,\beta\in\HH}c_{\alpha,\beta}
\left(\sum_{\gamma\in\HH}a_{\gamma,\alpha}X_{\gamma,\beta}\right)
\pmod{\T_{n,p+1}}.
\]
The coefficient of \(X_{\gamma,\beta}\) in \(ez_{(p)}\) is
\[
\sum_{\alpha\in\HH}a_{\gamma,\alpha}c_{\alpha,\beta},
\]
which is precisely the \((\gamma,\beta)\)-entry of the matrix product \(AC\).

Similarly, modulo \(\T_{n,p+1}\), right multiplication by \(e\) acts only on
the bottom half-diagram. Thus
\[
X_{\alpha,\beta}e
\equiv
\sum_{\gamma\in\HH}a_{\gamma,\beta}X_{\alpha,\gamma}
\pmod{\T_{n,p+1}}.
\]
Hence
\[
z_{(p)}e
\equiv
\sum_{\alpha,\beta\in\HH}c_{\alpha,\beta}
\left(\sum_{\gamma\in\HH}a_{\gamma,\beta}X_{\alpha,\gamma}\right)
\pmod{\T_{n,p+1}}.
\]
The coefficient of \(X_{\alpha,\gamma}\) in \(z_{(p)}e\) is
\[
\sum_{\beta\in\HH}c_{\alpha,\beta}a_{\gamma,\beta},
\]
which is precisely the \((\alpha,\gamma)\)-entry of \(CA^T\). It follows that
\[
ez_{(p)}\equiv z_{(p)}e \pmod{\T_{n,p+1}}
\]
if and only if
\begin{equation}\label{commute}
	AC=CA^T.
\end{equation}
The same condition must hold for each generator \(e_i\).

We now produce a solution of \eqref{commute} with the aid of the Gram matrix \(G\) of the bilinear form.
Specifically, we have
\[
G=(\langle \alpha,\beta\rangle)_{\alpha,\beta\in\HH}.
\]
Since the bilinear form is \(\TL_n(\delta)\)-invariant, we have
\[
\langle e\cdot\alpha,\beta\rangle=\langle \alpha,e\cdot\beta\rangle .
\]
Therefore
\[
\left\langle
\sum_{\gamma\in\HH}a_{\gamma,\alpha}\gamma,\beta
\right\rangle
=
\sum_{\gamma\in\HH}a_{\gamma,\alpha}\langle \gamma,\beta\rangle ,
\]
which is the \((\alpha,\beta)\)-entry of \(A^TG\). Similarly, we obtain
\[
\left\langle
\alpha,\sum_{\gamma\in\HH}a_{\gamma,\beta}\gamma
\right\rangle
=
\sum_{\gamma\in\HH}a_{\gamma,\beta}\langle \alpha,\gamma\rangle ,
\]
which is the \((\alpha,\beta)\)-entry of \(GA\). 

Therefore the \(\TL_n(\delta)\)-invariance property of the bilinear form is equivalent to the matrix equation
\[
A^TG=GA.
\]

If \(\VV\) is irreducible, then \(G\) is invertible. Hence
\[
A^TG=GA
\quad\Longrightarrow\quad
AG^{-1}=G^{-1}A^T.
\]
Thus \(G^{-1}\) is a nonzero solution of the equations
\[
A_iC=CA_i^T
\qquad (i=1,\ldots,n-1).
\]
In the irreducible case, these equations have a one-dimensional solution
space. Indeed, they say that the coefficient matrix \(C\) defines an
intertwining map \(\VV^*\to \VV\), and this space is one-dimensional by
Schur's lemma. Therefore, by \eqref{commute}, the leading term of
the unique nonzero class in \(Z_{n,p}/Z_{n,p+1}\) is, up to scalar,
\[
z_{(p)}
=
\sum_{\alpha,\beta\in\HH}
(G^{-1})_{\alpha,\beta}X_{\alpha,\beta}.
\]

Now suppose that \(\VV\) is reducible at a fixed non-generic value
\(\delta_0\in\Bbbk^\times\). Let \(u\) be an indeterminate. The matrices from the previous calculation now have entries in \(\Bbbk(u)\) and we write these as \(G(u)\) and \(A_i(u)\) to show their dependence on \(u\). As \(u\) is an indeterminate, the Gram matrix is invertible, and the preceding argument gives
\[
A_i(u)G(u)^{-1}=G(u)^{-1}A_i(u)^T
\]
for all \(i\).

At the non-generic value \(u=\delta_0\), the module \(V_{n,p}\) is reducible,
so the Gram matrix \(G(\delta_0)\) is singular. Thus the entries of
\(G(u)^{-1}\) may have poles at \(u=\delta_0\). Let \(r\) be the largest pole order at \(u=\delta_0\) among the entries of
\(G(u)^{-1}\). Define
\[
C_0
=
\left.(u-\delta_0)^rG(u)^{-1}\right|_{u=\delta_0}.
\]
Then \(C_0\) is nonzero, and multiplying the preceding identity by
\((u-\delta_0)^r\) and setting \(u=\delta_0\) gives
\[
A_i(\delta_0)C_0=C_0A_i(\delta_0)^T
\]
for all \(i\).

Writing \(C_0=((C_0)_{\alpha,\beta})_{\alpha,\beta\in\HH}\), the specialized
Gram-matrix construction gives
\[
\widetilde z_{(p)}=
\sum_{\alpha,\beta\in\HH}
(C_0)_{\alpha,\beta}X_{\alpha,\beta}
\]
as an element of the leading filtration quotient.
This element satisfies the same leading commutation relations as the leading
term of a central element. Theorem~\ref{thm:main} shows that the actual quotient
\(Z_{n,p}/Z_{n,p+1}\) is one-dimensional, but it does not identify its leading
term with this specialization of the inverse Gram matrix. It is therefore
natural to ask whether this specialization always detects
the central class, or whether further lower-filtration corrections are needed
to see the connection.

\end{document}